\documentclass[reqno,a4paper]{amsart}
\usepackage{amssymb,setspace}
\usepackage{ifpdf}
\ifpdf
 \usepackage[hyperindex,pagebackref]{hyperref}
\else
 \expandafter\ifx\csname dvipdfm\endcsname\relax
 \usepackage[hypertex,hyperindex,pagebackref]{hyperref}
 \else
 \usepackage[dvipdfm,hyperindex,pagebackref]{hyperref}
 \fi
\fi
\allowdisplaybreaks[4]
\numberwithin{equation}{section}
\theoremstyle{plain}
\newtheorem{theorem}{Theorem}[section]
\newtheorem{lemma}{Lemma}[section]
\theoremstyle{remark}
\newtheorem{remark}{Remark}[section]
\DeclareMathOperator{\td}{d\mspace{-1mu}}

\begin{document}

\title{Some inequalities for complete elliptic integrals}

\author[L. Yin]{Li Yin}
\address[Yin]{Department of Mathematics, Binzhou University, Binzhou City, Shandong Province, 256603, China}
\email{\href{mailto: L. Yin<yinli_79@163.com>}{yinli\_79@163.com}}

\author[F. Qi]{Feng Qi}
\address[Qi]{School of Mathematics and Informatics, Henan Polytechnic University, Jiaozuo City, Henan Province, 454010, China}
\email{\href{mailto: F. Qi<qifeng618@gmail.com>}{qifeng618@gmail.com}, \href{mailto: F. Qi<qifeng618@hotmail.com>}{qifeng618@hotmail.com}, \href{mailto: F. Qi<qifeng618@qq.com>}{qifeng618@qq.com}}
\urladdr{\url{http://qifeng618.wordpress.com}}



\thanks{The first author was supported partially by the Natural Science Foundation of Shandong Province under grant number ZR2012AQ028}

\begin{abstract}
In the paper, by using Lupa\c{s} integral inequality, the authors find some new inequalities for the complete elliptic integrals of the first and second kinds. These results improve some known inequalities.
\end{abstract}

\maketitle

\section{Introduction}

Legendre's complete elliptic integrals of the first and second kind are defined for real numbers $0<r<1$ by
\begin{equation}
\kappa(r)=\int_0^{\pi /2} {\frac{1} {{\sqrt{1-r^2\sin^2t}\,}}} \td t=\int_0^1 {\frac{1} {{\sqrt{(1-t^2 )(1-r^2 t^2 )}\,}}}
\td t
\end{equation}
and
\begin{equation}
\varepsilon(r)=\int_0^{\pi /2} {\sqrt{1-r^2\sin^2t}\,} \td t =
\int_0^1 {\sqrt{\frac{{1-r^2 t^2}} {{1-t^2}}}\,} \td t
\end{equation}
respectively. They can also defined by
\begin{equation}
\kappa(r,s)=\int_0^{\pi /2} {\frac{1} {{\sqrt{r^2\cos^2t +
s^2\sin^2t}\,}}} \td t
\end{equation}
 and
\begin{equation}
\varepsilon(r,s)=\int_0^{\pi /2} {\sqrt{r^2\cos^2t+s^2\sin
^2 t}\,} \td t.
\end{equation}
Letting $r'=\sqrt{1-r^2}\,$. We often denote
$$
\kappa'(r)=\kappa(r')\quad \text{and}\quad \varepsilon'(r)=\varepsilon(r').
$$
These integrals are special cases of Gauss hypergeometric
function
$$
F(a,b;c;x)=\sum_{n=0}^\infty{\frac{{(a,n)(b,n)}} {{(c,n)}}\frac{{x^n}} {{n!}}},
$$
where $(a,n)=\prod_{k=0}^{n-1} {(a+k)}$. Indeed, we have
$$
\kappa(r)=\frac{\pi}{2}F\biggl(\frac{1} {2},\frac{1}{2};1;r^2 \biggr) \quad\text{and}\quad
\varepsilon(r)=\frac{\pi}{2}F\biggl(-\frac{1} {2},\frac{1}{2};1;r^2 \biggr).
$$
For more information on the history, background, properties, and applications, please refer to~\cite{avv, avz} and related references.
\par
Recently, some bounds for $\varepsilon(r)$ and $\kappa(r)$ were discovered in the paper~\cite{gq}. For example, Theorem~1 in~\cite{gq} states that, for $0<r<1$,
\begin{equation}\label{guo-qi-ellp-ineq-log}
\frac{\pi} {2}-\frac{1} {2}\ln \frac{{(1+r)^{1-r}}} {{(1-
r)^{1+r}}} < \varepsilon(r) < \frac{{\pi-1}} {2}+\frac{{1-
r^2}} {{4r}}\ln \frac{{1+r}} {{1-r}}.
\end{equation}
For more information on inequalities of complete elliptic integrals, please refer to~\cite{a, adv, aqv, qh} and a short survey in~\cite[pp.~40\nobreakdash--46]{refine-jordan-kober.tex-JIA}.
\par
Motivating by the double inequality~\eqref{guo-qi-ellp-ineq-log}, some estimates for $\varepsilon(r)$ in terms of rational functions of the arithmetic, geometric, and roots square means were obtained in~\cite{cwqj, wc, wcqj}.

The aim of this paper is to establish some new inequalities for the complete elliptic integrals.

\section{Lemmas}

In order to prove our main results, the following lemma is necessary.

\begin{lemma}[Lupa\c{s} integral inequality~{\cite[p.~57]{bd}}]\label{yin-qi-lem2.1}
If $f',g'\in{L_{2}[a,b]}$, then
\begin{multline}
\biggl| {\frac{1} {{b-a}}\int_a^b {f(t)g(t)\td t}-\frac{1} {{b-
a}}\int_a^b {f(t)\td t\frac{1} {{b-a}}\int_a^b {g(t)\td t}}} \biggr|\\*
\le \frac{{b-a}} {\pi^2}\bigl\| f' \bigr\|_2 \bigl\|g' \bigr\|_2,
\end{multline}
where
$$
\bigl\| f' \bigr\|_2 =\biggl( {\int_a^b {\bigl| {f'(t)}
\bigr|^2 \td t}} \biggr)^{1/2} \quad \text{and} \quad \bigl\| g' \bigr\|_2 =
\biggl( {\int_a^b {\bigl| {g'(t)} \bigr|^2 \td t}} \biggr)^{1/2}.
$$
\end{lemma}

\section{Main Results}

Now we are in a position to find some inequalities for complete elliptic integrals.

\begin{theorem}
For $r\in(0,1)$, we have
\begin{equation}\label{(3.1)}
\frac{{\pi \sqrt{6+2\sqrt{1-r^2}\,-3r^2}\,}} {{4\sqrt2\,}}
\le \varepsilon(r) \le \frac{{\pi \sqrt{10-2\sqrt{1
- r^2}-5r^2}\,}} {{4\sqrt2\,}} .
\end{equation}
\end{theorem}

\begin{proof}
Taking
$$
{f(t)=g(t)=\sqrt{1-r^2\sin^2t}\,}
$$
and letting $a=0$ and $b=\frac{\pi}{2}$ in Lemma~\ref{yin-qi-lem2.1} yield
\begin{multline}\label{(3.2)}
\biggl| {\frac{2}{\pi}\int_0^{\pi /2} \bigl(1-r^2\sin^2t\bigr) \td t-\frac{4} {\pi^2}\varepsilon^2(r)}\biggr|
=\biggl| {\frac{4} {\pi^2}\varepsilon^2(r)-\frac{2-r^2}{2}} \biggr| \\
\le \frac{1}{{2\pi}}\int_0^{\pi /2} {\frac{{r^4\sin^2t\cos^2t}} {{1-r^2\sin^2t}}} \td t
=\frac{1}{{2\pi}}\int_0^{\pi /2} {r^4\sin^2t\cos^2t\sum_{n=0}^\infty {r^{2n}\sin^{2n} t}} \td t \\
=\frac{1}{{2\pi}}\sum_{n=0}^\infty {\int_0^{\pi /2} {r^{2n+4}\bigl(\sin^{2n+2} t-\sin^{2n+4} t\bigr)} \td t}
=\frac{1}{4}h(r),
\end{multline}
where, by
\begin{equation}\label{int-sine-i-power}
\int_0^{\pi /2} {\sin^{2i} t} \td t=\frac{\pi}{2}\frac{{(2i-1)!!}} {{(2i)!!}}
\end{equation}
for $i\in\mathbb{N}$,
$$
h(r)=\sum_{n=0}^\infty {\frac{{(2n+1)!!}}
{{(2n+2)!!}}\frac{{r^{2n+4}}} {{2n+4}}}.
$$
A direct calculation yields
\begin{equation*}
h'(r)=\sum_{n =0}^\infty {\frac{{(2n+1)!!}} {{(2n+2)!!}}r^{2n+3} =
r\sum_{n=0}^\infty {\frac{{(2n+1)!!}} {{(2n +
2)!!}}r^{2n+2} =r\biggl( {\frac{1} {{\sqrt{1-r^2}\,}}-1}
\biggr)}},
\end{equation*}
where we used
\begin{equation}
\frac{1} {{\sqrt{1-t^2}\,}}=\sum_{n=0}^\infty{\frac{{(2i-1)!!}} {{(2i)!!}}t^{2n}},\quad |t|<1.
\end{equation}
Hence, we have
\begin{equation}
h(r)=h(0)+\int_0^r {h'(r)} \td r=1-\sqrt{1-r^2}\,-\frac{{r^2}} {2}.
\end{equation}
Substituting this equality into~\eqref{(3.2)} gives
$$
\biggl| {\frac{4} {{\pi^2}}\varepsilon^2(r)-\frac{2-r^2} {2}} \biggr|
\le\frac{1} {4}-\frac{{\sqrt{1-r^2}\,}} {4}-\frac{{r^2}} {8}.
$$
This means the double inequality~\eqref{(3.1)}.
\end{proof}

\begin{remark}
By the well-known software \textsc{Mathematica}, we can show
that
\begin{enumerate}
\item
the left-hand side inequality in~\eqref{(3.1)} refines the corresponding
one in~\eqref{guo-qi-ellp-ineq-log};
\item
the right-hand side inequalities in~\eqref{(3.1)} and~\eqref{guo-qi-ellp-ineq-log} are not contained each other;
\item
when $r\in[\frac{1}{4},\frac{3}{4}]$, the right-hand side inequality in~\eqref{(3.1)} is better than the corresponding one in~\eqref{guo-qi-ellp-ineq-log}.
\end{enumerate}
\end{remark}

\begin{theorem}
For $r\in(0,1)$, we have
\begin{equation}\label{(3.5)}
\frac{{\pi \sqrt{32-r^4-32r^2}\,}} {{8\sqrt2\, \sqrt[4]{{(1-
r^2 )^3}}\,}} \le \kappa(r) \le \frac{{\pi \sqrt{r^4-
32r^2+32}\,}} {{8\sqrt2\, \sqrt[4]{{(1-r^2 )^3}}\,}}.
\end{equation}
\end{theorem}

\begin{proof}
Taking
$$
{f(t)=g(t)=\frac{1}{\sqrt{1-r^2\sin^2t}\,}}
$$
and letting $a=0$ and $b=\frac{\pi}{2}$ in Lemma~\ref{yin-qi-lem2.1} lead to
\begin{gather*}
\biggl| {\frac{2}{\pi}\int_0^{\pi /2} {\frac{1} {{1-r^2\sin^2t}}} \td t-\frac{4}{\pi^2}\kappa^2(r)} \biggr|
=\Biggl| {\frac{4} {{\pi^2}}\kappa^2(r)-\frac{2}{\pi}\int_0^{\pi /2} {\sum_{n=0}^\infty {r^{2n}\sin^{2n} t}} \td t} \Biggr| \\
=\Biggl| {\frac{4}{\pi^2}\kappa^2(r)-\sum_{n=0}^\infty {\frac{{(2n-1)!!}} {{(2n)!!}}r^{2n}}} \Biggr|
=\biggl|{\frac{4} {\pi^2}\kappa^2(r)-\frac{1}{{\sqrt{1-r^2}\,}}} \biggr| \\
\le \frac{1}{{2\pi}}\int_0^{\pi /2} {\frac{{r^4\sin^2t\cos^2t}} {{\bigl(1-r^2\sin^2t\bigr)^3}}} \td t\\
=\frac{1}{{4\pi}}\int_0^{\pi /2} {r^4\sin^2t\cos^2t\sum_{n=0}^\infty{(n+2)(n+1)r^{2n}\sin^{2n} t}}\td t\\
=\frac{1}{8}\sum_{n=0}^\infty {(n+2)(n+1)r^{2n+4} \frac{{(2n+ 1)!!}} {{(2n+2)!!}}\frac{1}{{2n+4}}}
=\frac{{r^3}}{{32}}p(r),
\end{gather*}
where
$$
p(r)=\sum_{n=0}^\infty {\frac{{(2n+2)(2n +1)!!}} {{(2n+2)!!}}r^{2n+1}}
$$
satisfies $p(0)=0$,
$$
\int_0^r {p(r)} \td r=\sum_{n=0}^\infty {\frac{{(2n +1)!!}} {{(2n+2)!!}}r^{2n+2}
=\frac{1} {{\sqrt{1-r^2}\,}}-1},
$$
and so
\begin{equation}
{p(r)=\frac{r} {{\sqrt{(1-r^2)^3}\,}}}.
\end{equation}
Consequently, we find
\begin{equation}
\biggl| {\frac{4} {\pi^2}\kappa^2(r)-\frac{1} {{\sqrt{1-r^2}\,}}} \biggr|
\le \frac{{r^4}}{{32\sqrt{(1-r^2 )^3}\,}}.
\end{equation}
The double inequality~\eqref{(3.5)} follows.
\end{proof}

\begin{theorem}
For $r>0$ and $s>0$, we have
\begin{equation}\label{(3.9)}
\frac{\pi} {8}\sqrt{\frac{{8rs(r^2+s^2 )-(s^2-r^2 )^2}}
{{rs}}}\, \le \varepsilon(r,s) \le \frac{\pi} {8}\sqrt
{\frac{{8rs(r^2+s^2 ) +(s^2-r^2 )^2}} {{rs}}}\, .
\end{equation}
\end{theorem}

\begin{proof}
Taking
$$
{f(t)=g(t)={\sqrt{r^2\cos^2t+s^2\sin^2t}\,}}
$$
and letting $a=0$ and $b=\frac{\pi}{2}$ in Lemma~\ref{yin-qi-lem2.1} reveal
\begin{gather*}
\biggl|{\frac{2}{\pi}\int_0^{\pi /2} {\bigl(r^2\cos^2t+s^2\sin^2t\bigr)} \td t -\frac{4}{\pi^2}\varepsilon^2(r,s)}\biggr|
=\biggl| {\frac{4}{\pi^2}\varepsilon^2(r,s)-\frac{{r^2+s^2}}{2}} \biggr| \\
\le\frac{1}{{2\pi}}\int_0^{\pi/2} {\frac{{(s^2-r^2)^2\sin^2t\cos^2t}}{{r^2\cos^2t+s^2\sin^2t}}} \td t
\le \frac{{(s^2-r^2 )^2}}{{2\pi}}\frac{1} {4}\int_0^{\pi /2} {\frac{1}{{r^2\cos^2t+s^2\sin^2t}}} \td t \\
 =\frac{{(s^2-r^2 )^2}}{{8\pi}}\frac{1} {{rs}}\arctan \biggl(\frac{s} {r}\tan t|^{\pi/2}_0\biggr)
=\frac{{(s^2-r^2 )^2}}{{16rs}}.
\end{gather*}
This means the double inequality~\eqref{(3.9)}.
\end{proof}

\begin{theorem}
For $s>r>0$, we have
\begin{equation}\label{(3.11)}
\biggl| {\frac{4} {\pi^2}\kappa^2(r,s)-\frac{1} {{\pi rs}}}\biggr|
\le \frac{{s^2-r^2}} {{32rs}}\biggl( {\frac{1} {{s^2}}+\frac{1} {{r^2}}} \biggr).
\end{equation}
\end{theorem}

\begin{proof}
Taking
$$
{f(t)=g(t)=\frac{1}{{\sqrt{r^2\cos^2t+s^2\sin^2t}\,}}}
$$
and letting $a=0$ and $b=\frac{\pi}{2}$ in Lemma~\ref{yin-qi-lem2.1} figure out
\begin{gather*}
 \biggl| {\frac{2}{\pi}\int_0^{\pi /2} {\frac{1} {{r^2\cos^2t+s^2\sin^2t}}}\td t-\frac{4}
{\pi^2}\kappa^2(r,s)} \biggr|
=\biggl| {\frac{4}{\pi^2}\kappa^2(r,s)-\frac{1}{{\pi rs}}} \biggr| \\
\le\frac{1}{{2\pi}}\int_0^{\pi/2}{\frac{{\bigl(s^2-r^2\bigr)^2\sin^2t\cos^2t}} {{\bigl(r^2\cos^2t+s^2\sin^2t\bigr)^3}}} \td t \\
 =\frac{{r^2-s^2}}{{8\pi}}\int_0^{\pi /2} {\sin t\cos t} \td\frac{1}
{\bigl(r^2\cos^2t+s^2\sin^2t\bigr)^2} \\
 =\frac{{r^2-s^2}}{{8\pi}}\Biggl[ {\int_0^{\pi /2} {\frac{{\sin^2t}} {{\big(r^2\cos^2
t+s^2\sin^2t\bigr)^2}}} \td t-\int_0^{\pi /2} {\frac{{\cos^2t}}
{\bigl(r^2\cos^2t+s^2\sin^2t\bigr)^2}} \td t} \Biggr] \\
 =\frac{{r^2-s^2}}{{8\pi}}\Biggl[ {\int_0^{\pi /2} {\frac{{\csc^2 t}} {{\bigl(r^2 \cot^2
t+s^2\bigr)^2}}} \td t-\int_0^{\pi /2} {\frac{{\sec^2 t}}
{{\bigl(r^2+s^2 \tan^2 t\bigr)^2}}} \td t} \Biggr] \\
 =\frac{{r^2-s^2}}{{8\pi}}\biggl[ {\int_0^\infty {\frac{1} {{(r^2 u^2+s^2 )^2}}}
\td u-\int_0^{\pi /2} {\frac{1}{{(r^2+s^2 \lambda^2 )^2}}} \td\lambda} \biggr]
=\frac{{s^2-r^2}}{{32rs}}\biggl( {\frac{1} {{s^2}}+\frac{1}{{r^2}}} \biggr).
\end{gather*}
The proof is complete.
\end{proof}

\begin{remark}
From~\eqref{(3.11)}, we easily obtain
\begin{equation}\label{(3.13)}
\frac{\pi} {2}\sqrt{\frac{{32r^2 s^2-\pi(s^4-r^4 )}} {{32\pi r^3 s^3}}}\,
\le \kappa(r,s) \le \frac{\pi} {2}\sqrt{\frac{{32r^2 s^2+\pi(s^4-r^4 )}} {{32\pi r^3 s^3}}}\, .
\end{equation}
By the software \textsc{Mathematica}, we can show that the double inequality in~\eqref{(3.13)} and the inequality in~\cite[Theorem~2]{gq} are not contained each other.
\end{remark}


\begin{thebibliography}{99}

\bibitem{a}
H. Alzer, \textit{Sharp inequalities for the complete elliptic integrals of the first kinds}, Math. Proc. Camb. Phil. Soc. \textbf{124} (1988), no.~2, 309\nobreakdash--314; Available online at \url{http://dx.doi.org/10.1017/S0305004198002692}.

\bibitem{adv}
G. D. Anderson, P. Duren, and M. K. Vamanamurthy, \textit{An inequality for complete elliptic integrals}, J. Math. Anal. Appl. \textbf{182} (1994), no.~1, 257\nobreakdash--259; Available online at \url{http://dx.doi.org/10.1006/jmaa.1994.1080}.

\bibitem{aqv}
G. D. Anderson, S.-L. Qiu, and M. K. Vamanamurthy, \textit{Elliptic integrals inequalities, with applications}, Constr. Approx. \textbf{14} (1998), no.~2, 195\nobreakdash--207; Available online at \url{http://dx.doi.org/10.1007/s003659900070}.

\bibitem{avv}
G. D. Anderson, M. K. Vamanamurthy, and M. Vuorinen, \textit{Topics in special functions}, available online at \url{http://arxiv.org/abs/0712.3856}.

\bibitem{avz}
G. D. Anderson, M. Vuorinen, and X.-H. Zhang, \textit{Topics in special functions III}, available online at \url{http://arxiv.org/abs/1209.1696}.

\bibitem{bd}
N. S. Barnett and S. S. Dragomir, \textit{On an inequality of the Lupa\c{s} type}, Demonstratio Math. \textbf{42} (2008), no.~1, 57\nobreakdash--62.

\bibitem{cwqj}
Y.-M. Chu, M.-K. Wang, S.-L. Qiu, and Y.-P. Jiang, \textit{Bounds for complete elliptic integrals of second kind with applications}, Comput. Math. Appl. \textbf{63} (2012), 1177\nobreakdash--1184; Available online at \url{http://dx.doi.org/10.1016/j.camwa.2011.12.038}.

\bibitem{gq}
B.-N. Guo and F. Qi, \textit{Some bounds for the complete elliptic integrals of the first and second kind}, Math. Inequal. Appl. \textbf{14} (2011), no.~2, 323\nobreakdash--334; Available online at \url{http://dx.doi.org/10.7153/mia-14-26}.

\bibitem{qh}
F. Qi and Z. Huang, \textit{Inequalities for complete elliptic integrals}, Tamkang J. Math. \textbf{29} (1998), no.~3, 165\nobreakdash--169; Available online at \url{http://dx.doi.org/10.5556/j.tkjm.29.1998.165-169}.

\bibitem{refine-jordan-kober.tex-JIA}
F. Qi, D.-W. Niu, and B.-N. Guo, \textit{Refinements, generalizations, and applications of Jordan's inequality and related problems}, J. Inequal. Appl. \textbf{2009} (2009), Article ID 271923, 52 pages; Available online at \url{http://dx.doi.org/10.1155/2009/271923}.

\bibitem{wc}
M.-K. Wang and Y.-M. Chu, \textit{Asymptotical bounds for complete elliptic integrals of second kind}, available online at \url{http://arxiv.org/abs/1209.0066}.

\bibitem{wcqj}
M.-K. Wang, Y.-M. Chu , S.-L. Qiu, and Y.-P. Jiang, \textit{Bounds for the perimeter of an ellipse}, J. Approx. Theory \textbf{164} (2012), 928\nobreakdash--937; Available online at \url{http://dx.doi.org/10.1016/j.jat.2012.03.011}.

\end{thebibliography}
\end{document}